\documentclass[11 pt ]{article}

\usepackage[margin=1in]{geometry}

\usepackage{tikz}
\usetikzlibrary{calc}

\usepackage[none]{hyphenat}
\usepackage[utf8]{inputenc}
\usepackage{float}
\usepackage{amsmath,amsthm,amsfonts,hyperref, amssymb, mathtools}
\usepackage{ dsfont }
\usepackage[utf8]{inputenc}
\usepackage[english]{babel}
\usepackage{graphicx}
\usepackage{mathtools}
\usepackage[export]{adjustbox}
\usepackage[ampersand]{easylist}
\usepackage[autostyle]{csquotes}
\usepackage{enumitem}
\usepackage{bbm}
\usepackage{wasysym}
\usepackage[style=alphabetic, maxalphanames=5, maxnames=5, backend=biber]{biblatex}
\usepackage{wrapfig}
\usepackage[mathscr]{eucal}
\usepackage{comment}

\usepackage[inkscapeformat=png]{svg}

\DeclareMathOperator{\area}{Area}
\let\sf\relax
\DeclareMathOperator{\sf}{SF}
\renewcommand{\geq}{\geqslant}
\renewcommand{\leq}{\leqslant}
\renewcommand{\epsilon}{\varepsilon}
\newcommand{\defeq}{\vcentcolon=}
\newcommand{\eqdef}{=\vcentcolon}

\newcommand{\sfr}{\mathfrak{S}}
\newcommand{\eps}{\varepsilon}

\newcommand\restr[2]{{
  \left.\kern-\nulldelimiterspace 
  #1 
  \littletaller 
  \right|_{#2} 
  }}

\newcommand{\littletaller}{\mathchoice{\vphantom{\big|}}{}{}{}}

\title{On the number of sum-free subsets of the square grid}
\author{Anubhab Ghosal
\thanks{Mathematical Institute, University of Oxford, UK. Supported by the Clarendon Fund.
Email: \textbf{ghosal@maths.ox.ac.uk}}.
}
\date{}

\newtheorem{theorem}{Theorem}[section]

\newtheorem{prop}[theorem]{Proposition}
\newtheorem{conj}[theorem]{Conjecture}
\newtheorem{lemma}[theorem]{Lemma}
\newtheorem{corollary}{Corollary}[theorem]

\theoremstyle{definition}
\newtheorem{definition}[theorem]{Definition}
\newtheorem{observation}[theorem]{Observation}

\begin{document}

\maketitle

\begin{abstract}
    Generalising the Cameron--Erd\H{o}s conjecture to two dimensions, Elsholtz and Rackham conjectured that the number of sum-free subsets of $[n]^2$ is $2^{0.6n^2+O(n)}$. We prove their conjecture.
\end{abstract}

\section{Introduction}

Let $G$ be an \emph{additive group}, that is, a group equipped with the commutative binary operation $+$. A set $S\subseteq G$ is \emph{sum-free} if there are no solutions to $a+b=c$ in $S$.

\subsection{Sum-free subsets of $[n]$}
 
Let $G=(\mathbb{Z},+)$ and consider $[n]\defeq\{1,\dots, n\}\subset G$. A classical puzzle asks one to determine the largest size of a sum-free subset of $[n]$. The answer is $\lceil\frac{n}{2}\rceil$, attained both by $\{\lfloor \frac{n}{2}\rfloor+1, \dots, n\}$ and the set of odd numbers in $[n]$. These are the only extremal examples and Freiman \cite{freiman1991structure} established \emph{stability}: any large enough sum-free subset of $[n]$ either consists entirely of odd numbers or is close to the second half of an interval.

Having gained some understanding of the large sum-free subsets of $[n]$, a natural next step is to investigate the number of sum-free subsets of $[n]$. For a finite subset $R\subseteq G$, let $M(R)$ denote the size of a maximal sum-free subset $S$ of $R$ and let $\sf(R)$ denote the number of sum-free subsets $S\subseteq R$.

\begin{observation}
\label{obs1}
    As the subsets of a maximal sum-free set are all sum-free,
    $\sf(R)\geq 2^{M(R)}$.
\end{observation}

\noindent Consequently, $\sf([n])$ is at least $2^{\frac{n}{2}}$. Cameron and Erd\H{o}s \cite{cameron1990number} conjectured that this is in fact the correct number up to a constant multiplicative factor, that is, $\sf([n])=O(2^{\frac{n}{2}})$.

In their paper, Cameron and Erd\H{o}s showed that the number of sum-free subsets of integers in $[\frac{n}{3}, n]$ is at most $O(2^{\frac{n}{2}})$. So, the remaining step was to bound the number of sum-free subsets of $[n]$ with an element smaller than $\frac{n}{3}$. This was done in 2003 by Green \cite{green2004cameron} and, independently, by Sapozhenko \cite{sapozhenko2008cameron}, thereby settling the celebrated Cameron--Erd\H{o}s conjecture. The crucial ingredient in both proofs was constructing a family of containers, which Green did using Fourier analysis, whereas Sapozhenko's methods were more combinatorial. These works were some of the first applications of the container method, which has since blossomed -- see \cite{balogh2018method} for a survey.

\subsection{Sum-free sets in two dimensions}

The above questions generalise naturally to the study of sum-free sets in the lattice $\mathbb{Z}^d$ under coordinatewise addition. In particular, the problem of determining $M([n]^2)$ was communicated by Harout Aydinian to Oriol Serra and posed as an open problem at the 19th British Combinatorial Conference in 2001 \cite{cameron2005research}. Cameron \cite{cameronnote} showed that $0.6n^2+O(n)\leq M([n]^2) \leq \frac{1}{\sqrt{e}}n^2+O(n)$ and conjectured \cite{cameron2005research, cameronnote} that, in fact, $M([n^2])=0.6 n^2+O(n)$. The following illustrates the construction used for the lower bound.

\begin{definition}
    For $n\in \mathbb{N}$, we refer to the set $\{(x,y)\in\mathbb{R}^2:0.8n\leq x+y < 1.6 n,0\leq x,y\leq n\}$ as \emph{the big stripe}.
\end{definition}

\begin{figure}[H]
    \centering
    \begin{tikzpicture}[scale=6]
  \def\n{1}
  \def\a{0.8}
  \def\b{1.6}

  \coordinate (A) at (0,0);
  \coordinate (B) at (\n,0);
  \coordinate (C) at (\n,\n);
  \coordinate (D) at (0,\n);

  \begin{scope}
    \clip (0,\a*\n) -- (\a*\n,0) -- (\n,0) -- (\n,\n) -- (0,\n) -- cycle;
    \clip (0,0) -- (0,\n) -- ({\b*\n-\n},\n) -- (\n,{\b*\n-\n}) -- (\n,0) -- cycle;
    \fill[orange!20] (0,0) rectangle (\n,\n);
  \end{scope}

  \draw[very thick] (A)--(B)--(C)--(D)--cycle;

  \draw[teal!70!black,thick]
    (0,{\a*\n}) -- ({\a*\n},0)
    node[midway,sloped,below,text=black] {$x+y=0.8n$};
  \draw[teal!70!black,thick]
    ({\b*\n-\n},\n) -- (\n,{\b*\n-\n})
    node[midway,sloped,below,text=black] {$x+y=1.6n$};

  \foreach \p/\pos/\txt in {
    A/below left/{(0,0)},
    B/below right/{(n,0)},
    C/above right/{(n,n)},
    D/above left/{(0,n)}
  }{
    \fill (\p) circle (0.012);
    \node[\pos] at (\p) {\(\txt\)};
  }
\end{tikzpicture}
    \caption{The set of lattice points in the big stripe is sum-free.}
\end{figure}
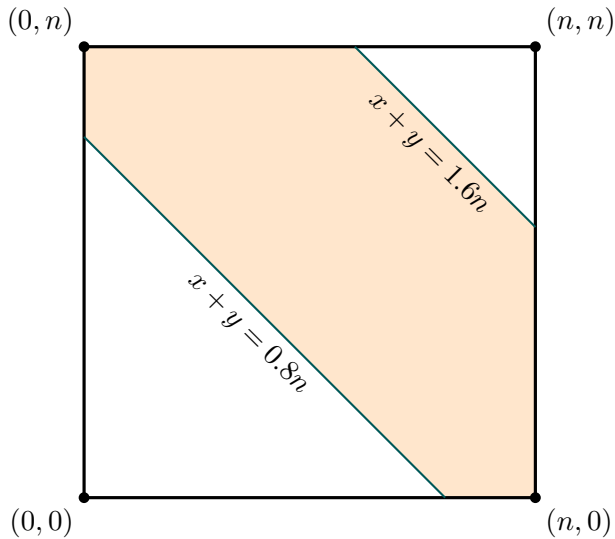

Elsholtz and Rackham \cite{elsholtz2017maximal} proved the following theorem in 2017, settling Cameron's conjecture.

\begin{theorem}[{\cite[Theorem 1.4]{elsholtz2017maximal}}]
\label{max2d}
    The maximal size of a sum-free subset of $[n]^2$ is $M([n]^2)=0.6n^2+O(n)$.
\end{theorem}

Stability of the above construction was recently established by Liu, Wang, Wilkes, and Yang \cite{liu2023shape}.

\begin{definition}
    A set $S\subseteq [n]^2$ is \emph{$\gamma$-close to the big stripe} if it lies entirely in the stripe $\{(x,y)\in\mathbb{R}^2:(0.8-\gamma)n\leq x+y \leq (1.6+\gamma)n\}$.
\end{definition}

\begin{theorem}[{\cite[Theorem 1.2]{liu2023shape}}]
\label{stability}
    For all $\gamma>0$, there exists $\delta>0$ and $n_0\in\mathbb{N}$ such that the following holds for all $n\geq n_0$.
    If $S\subseteq [n]^2$ is sum-free with $|S|\geq (0.6-\delta)n^2$, then $S$ is $\gamma$-close to the big stripe.
\end{theorem}

Generalising the Cameron--Erd\H{o}s conjecture to two dimensions, Elsholtz and Rackham made the following conjecture on the number of sum-free subsets of $[n]^2$.

\begin{conj}[{\cite[Conjecture 2, $d=2$]{elsholtz2017maximal}}]
\label{elsholtzrackhamconjd=2}
    The number of sum-free subsets of $[n]^2$ is $2^{0.6n^2+O(n)}$.
\end{conj}

\noindent The lower bound in Conjecture \ref{elsholtzrackhamconjd=2} follows from Theorem \ref{max2d} and Observation \ref{obs1}. We prove the following upper bound, settling Conjecture \ref{elsholtzrackhamconjd=2}.

\begin{theorem}
\label{mainthm}
    The number of sum-free subsets of $[n]^2$ is at most $2^{0.6n^2+O(n)}$.
\end{theorem}

\noindent When $5$ divides $n$, one can check that the big stripe has $0.6n^2+0.6n-2$ lattice points inside. Therefore, by Observation \ref{obs1}, for general $n$, Theorem \ref{mainthm} is tight up to the implicit constant.

\subsection{Higher dimensions}

The extremal problem has also been raised in higher dimensions by Aydinian and Cameron \cite{cameronnote} and was highlighted recently by Green \cite[Problem 6]{green100open}. The following folklore conjecture asserts that the analogous construction in higher dimensions remains optimal.

\begin{conj}
\label{maxdd} The maximal size of a sum-free subset of $[n]^d$ is $M([n]^d)=(k_d+o(1))n^d$, where $k_d$ denotes the maximal volume of a region in the solid unit cube $[0,1]^d$ bounded between a hyperplane $H$ and its dilate $2\cdot H$.
\end{conj}
 
\noindent A stronger version of the conjecture would stipulate an error term of $O_d(n^{-1})$. In a recent breakthrough, Lepsveridze and Sun \cite{lepsveridze2023size} established Conjecture \ref{maxdd} for $d\in\{3,4,5\}$. For yet larger $d$, the problem remains open.


The Elsholtz--Rackham generalisation of the Cameron--Erd\H{o}s conjecture reads as follows.

\begin{conj}[{\cite[Conjecture 2]{elsholtz2017maximal}}]
\label{elsholtzrackhamconj}
    The number of sum-free subsets of $[n]^d$ is $2^{k_dn^d+O_d(n^{d-1})}$, where $k_d$ is as defined in the statement of Conjecture \ref{maxdd}.
\end{conj}

\noindent We will establish the following weaker version of the conjecture. 

\begin{prop}
\label{propapproxcounting}
    The number of sum-free subsets of $[n]^d$ is $2^{M([n]^d)+o(n^d)}$.
\end{prop}

\noindent The case $d=1$ of Proposition \ref{propapproxcounting} was established independently by Alon \cite{alon1991independent}, Calkin \cite{calkin1990number} and Erd\H{o}s and Granville (unpublished, see \cite{green2004cameron}). The proof of the proposition is nowadays a standard application of the hypergraph container method, developed independently by Saxton and Thomason \cite{saxton2015hypergraph}, and Balogh, Morris and Samotij \cite{balogh2015independent}.

Finally, we note that Conjecture \ref{elsholtzrackhamconj} remains completely out of reach for $d\geq 6$ until Conjecture \ref{maxdd} is proven. For $d\in\{3,4,5\}$, a possible approach towards resolving the Elsholtz--Rackham Conjecture is to establish the earlier mentioned stronger version of Conjecture \ref{maxdd} and stability in conjunction with the methods developed in this paper.

\subsection{Proof strategy and organisation}
We split Theorem \ref{mainthm} into two parts, Propositions \ref{proximity} and \ref{dircount}.

First, in section \ref{close}, we show that almost all sum-free subsets are close to the extremal example. This is done using the hypergraph container method. In particular, we use a container lemma in combination with Green's removal lemma \cite{green2005szemeredi} and the stability result. Along the way, we also deduce Proposition \ref{propapproxcounting} as a quick application of the machinery developed in this section.

Next, in section \ref{far}, we count the number of sum-free subsets close to the extremal example. We reduce the problem to bounding the number of sum-free subsets of a simpler region. We then divide the new region into triples of 1-dimensional fibers. First, we consider each triple separately to get a preliminary bound, and then improve this bound by considering interaction between triples.\newline

\noindent \textbf{Acknowledgements.}
The research leading to this work was carried out as a summer project in 2023 under the supervision of Yifan Jing and Akshat Mudgal supported by departmental funding from the Mathematical Institute, University of Oxford. The author is thankful to Yifan and Akshat for supervising the project and to the department for funding it. The author is also grateful to Oliver Riordan for being a supportive mentor and carefully reviewing drafts of the paper. The author thanks Mehtaab Sawhney for suggesting the inclusion of Proposition \ref{propapproxcounting} in the paper. Finally, the author also thanks Zach Hunter, Peter Keevash, Rob Morris, Ritesh Goenka, Ashwat Jain, and Reemon Spector for helpful discussions.

\section{Approximate typical structure}
\label{close}

In this section, our main aim is to prove the following proposition about the proximity of almost all sum-free sets to the big stripe. Along the way, we will also prove Proposition \ref{propapproxcounting}.

\begin{prop}
\label{proximity}
    The number of sum-free subsets of $[n]^2$ that are not $0.1$-close to the big stripe is $o(2^{0.6n^2})$.
\end{prop}

\subsection{Containers}

A family of sets so that any sum-free set is a subset of some set in the family is referred to as a family of \textit{containers} for sum-free sets. We will use a hypergraph container theorem to construct a suitable family of containers. But first, we must recall some terminology about hypergraphs.

A $r$-uniform hypergraph $\mathcal{H}$ consists of the set of its \textit{vertices} $V(\mathcal{H})$ and a set of \textit{edges} $E(\mathcal{H})$ so that each edge is a set of of $r$ distinct vertices, that is, $E(\mathcal{H})\subseteq\binom{V(\mathcal{H})}{r}$. Let $v(\mathcal{H})=|V(\mathcal{H})|$ and $e(\mathcal{H})=|E(\mathcal{H})|$ denote the number of vertices and edges respectively in $\mathcal{H}$. For $v\in V(\mathcal{H})$, its \textit{degree} $d(v)$ is the number of edges containing $v$. Therefore, the average degree of $\mathcal{H}$ is $\frac{r\cdot e(\mathcal{H})}{v(\mathcal{H})}$.

Let $\mathcal{H}$ be a $r$-uniform hypergraph with average degree $\bar{d}$. For every $S \subseteq V(\mathcal{H})$, its \textit{co-degree} $d(S)$ is the number of edges in $\mathcal{H}$ containing $S$, that is, $d(S)=|\{e\in E(\mathcal{H}):S\subseteq e\}|$. For every $j \in [r]$, denote by $\Delta_j$ the $j$-th maximum co-degree, that is, $\Delta_j=\max\{d(S):S\subseteq V(\mathcal{H}), |S|=j\}$. For $\tau \in (0,1)$, define $\Delta(\mathcal{H},\tau)=2^{\binom{r}{2}-1}\sum_{j=2}^r\frac{\Delta_j}{\bar{d}\tau^{j-1}2^{\binom{j-1}{2}}}$. In particular, when $r=3$,
$$\Delta(\mathcal{H},\tau)=\frac{4\Delta_2}{\bar{d}\tau}+\frac{2\Delta_3}{\bar{d}\tau^2}.$$

We will use the following version of the hypergraph container theorem, stated in this form in \cite{balogh2018number} but originally from \cite{saxton2015hypergraph}.

\begin{lemma}
\label{contheo}
    Let $\mathcal{H}$ be an $r$-uniform hypergraph with vertex set $[N]$. Let $0<\epsilon, \tau<1/2$. Suppose that $\tau<1/(200\cdot r \cdot r!^2)$ and $\Delta(\mathcal{H},\tau)\leq \epsilon/(12r!)$. Then there exists $c = c(r) \leq 1000 \cdot r\cdot r!^3$ and a collection $\mathcal{C}$ of vertex subsets such that
    \begin{enumerate}[label=(\roman*)]
        \item every independent set in $\mathcal{H}$ is a subset of some $A \in \mathcal{C}$;
        \item for every $A \in \mathcal{C}$, $e(\mathcal{H}[A]) \leq \epsilon \cdot e(\mathcal{H})$;
        \item $\log |\mathcal{C}|\leq cN\tau \cdot \log (1/\epsilon)\cdot\log(1/\tau)$.
    \end{enumerate}
\end{lemma}

We want a small family of containers with each container being \enquote{almost} sum-free, that is, each container has few \textit{Schur triples}.

\begin{definition}
    A Schur triple is a solution $(a,b,c)$ to $a+b=c$.
\end{definition}

We are now ready to construct our containers.

\begin{lemma}
\label{conts}
    For all  $d\in\mathbb{N}$ and $\epsilon\in(0,1)$, there exists $n_0\in\mathbb{N}$ such that the following holds for all $n\geq n_0$. There exists a family $\mathcal{C}$ of subsets of $[n]^d$ with the following properties.
    \begin{enumerate}[label=(\roman*)]
        \item \label{contitem1} If $S\subseteq [n]^d$ is sum-free, then $S$ is contained in some member of $\mathcal{C}$.
        \item \label{contitem2}Every member of $\mathcal{C}$ has at most $\epsilon n^{2d}$ Schur triples.
        \item \label{contitem3}$|\mathcal{C}|=2^{o(n^d)}$.
    \end{enumerate}
\end{lemma}
\begin{proof}
    Let $\mathcal{H}$ be the $3$-uniform hypergraph encoding Schur triples in $[n]^d$, that is, let $V(\mathcal{H})=[n]^d$ and $E(\mathcal{H})=\{\{a,b,c\}\in \binom{V(\mathcal{H})}{3}:a+b=c\}$. The average degree $\bar{d}$ of $\mathcal{H}$ is $\bar{d}=\frac{3e(\mathcal{H})}{v(\mathcal{H})}=\Theta(n^d)$ and $\Delta_2(\mathcal{H})=\Delta_3(\mathcal{H})=O(1)$. Let $\tau=n^{-\frac{d}{4}}$. Then, $\Delta(\mathcal{H}, \tau)=O(n^{-\frac{d}{2}})$. We can now apply Lemma \ref{contheo} to the hypergraph $\mathcal{H}$. For $n$ large enough, in terms of $\epsilon$, the hypotheses of the lemma hold and $(i)$, $(ii)$ and $(iii)$ follow from their respective counterparts in Lemma \ref{contheo}, noting that $e(\mathcal{H})\leq n^{2d}$ and $c(3)n^{d}\tau \cdot \log (1/\epsilon)\cdot\log(1/\tau)=O(n^\frac{3d}{4}\log n)=o(n^{d})$.
\end{proof}

\subsection{Green's Removal Lemma}

We begin by stating the following lemma about almost sum-free sets in $[n]^d$. It follows immediately from Green's Removal Lemma \cite[Theorem 1.5]{green2005szemeredi}, by taking $G=(\mathbb{Z}/2n\mathbb{Z})^d$, $k=3$, $A_1=A_2=A$ and $A_3=-A$ (See the proof of Corollary 1.6 in \cite{green2005szemeredi}).

\begin{lemma}[\cite{green2005szemeredi}]
\label{removal}
    For all $d\in \mathbb{N}$ and $\beta>0$, there exists $\epsilon>0$ and $n_0\in\mathbb{N}$ such that the following holds for all $n\geq n_0$.
    Suppose $A\subseteq [n]^d$ has at most $\epsilon n^{2d}$ Schur triples. Then, $A=B\cup C$, where $B$ is sum-free and $|C|\leq \beta n^d$.
\end{lemma}

We are now ready to prove Proposition \ref{propapproxcounting}.

\begin{proof}[Proof of Proposition \ref{propapproxcounting}]
    Let $\beta>0$ and let $\eps=\min(\frac{1}{2},\eps(\beta))$, where $\eps(\beta)$ is given by Lemma \ref{removal}. Apply Lemma \ref{conts} with parameter $\eps$ to get a family of containers $\mathcal{C}$. By property \ref{contitem2} of the containers and Lemma \ref{removal}, for $n\geq n_0(\beta)$, every $A\in \mathcal{C}$ can be written as $A=B\cup C$, where $B$ is sum-free and $|C|\leq \beta n^d$. As $B$ is sum-free, $|B|\leq M([n]^d)$ and so $|A|\leq M([n]^d)+\beta n^d$. By property \ref{contitem1} of the containers, for $n\geq n_0(\beta)$, the number of sum-free subsets of $[n]^d$ is at most $|\mathcal{C}|\max_{A\in\mathcal{C}}2^{|A|}$ which is at most $2^{M([n]^d)+\beta n^d+o(n^d)}$, by property \ref{contitem3} and the previous discussion. As $\beta>0$ was arbitrary, we are done.
\end{proof}

\subsection{Proof of Proposition \ref{proximity}}
\label{subsection proof of close}

For a region $R\subseteq \mathbb R^2$, let $\Lambda (R)\defeq\mathbb Z^2\cap R$ denote the set of lattice points in $R$ and let $\lambda(R)\defeq|\Lambda(R)|$ denote their number. The following lemma, a corollary of Lemma 2.1 in \cite{liu2023shape}, allows us to use the area to count lattice points. We shall make repeated use of it in this section, often without explicit reference.

\begin{lemma}
\label{estarea}
    For a region $R\subseteq [-2n,2n]^2$ which is a union of $O(1)$ convex polygons each with a finite number of sides, $\lambda(R)=\area(R)+O(n)$.
\end{lemma}

We are now ready to prove the main proposition of this section.

\begin{proof}[Proof of Proposition \ref{proximity}]
    Let $\gamma=\frac{1}{800}\cdot\log_2(\frac{4}{3})$ and let $\delta=\delta(\gamma)>0$ be given by Theorem \ref{stability}. Now, fix $\beta=\min(\frac{\delta}{2},\gamma)$ and let $\epsilon=\min(\frac{1}{2},\epsilon(\beta))$ where $\epsilon(\beta)$ is given by Lemma \ref{removal}. We can assume that $n$ is larger than $n_0$ stipulated in any of the hypotheses of Lemmas \ref{stability}, \ref{removal} and \ref{conts}.
    
    We apply Lemma \ref{conts} to get a family of containers $\mathcal{C}$. By the removal lemma, for $A\in \mathcal{C}$, $A=B\cup C$, where $B$ is sum-free and $|C|\leq \beta n^2$. Call the container $A$ \emph{small} if $|B|\leq (0.6-\delta)n^2$, and \emph{large} otherwise.

    We count the number of sum-free sets described by small containers first. A small container $A$ has at most $2^{|A|}\leq 2^{(0.6-\delta+\beta)n^2}$ sum-free subsets. Therefore, the number of sum-free subsets described by small containers is at most $2^{(0.6-\delta+\beta+o(1))n^2}=o(2^{0.6n^2})$.
    
    Now, let $A$ be a large container. By the stability result, $B$ is $\gamma$-close to the big stripe.
    
    Suppose a subset $S$ of $A$ contains a point $(a,b)$ satisfying $a+b\leq 0.7n$. Let
    $$R_1=\{(x,y)\in \mathbb{R}^2:x,y\geq 0.7n, 1.5n\leq x+y\leq1.6n\}.$$ Note that $R_1$ has area $\frac{3}{200}n^2$, and that $R_1$ and $R_1-(a,b)$ are contained in the big stripe. Therefore, one can find at least $\frac{\lambda(R_1)}{2}$ disjoint pairs of the form $\{(x,y), (x,y)-(a,b)\}$ inside the big stripe, so that $S$ can have at most one element from each pair.  As $A$ has at most $(2\gamma+\beta)n^2$ points outside the big stripe, the number of such $S\subseteq A$ is at most $2^{(0.6+2\gamma+\beta)n^2}\cdot(\frac{3}{4})^{\frac{\lambda(R_1)}{2}}$. Taking a union bound over all possible containers $A$, we get that the number of sum-free subsets of $[n]^2$ that have a point with coordinate sum less than $0.7 n$ is at most $2^{(0.6+2\gamma+\beta-\frac{3}{400}\cdot\log_2(\frac{4}{3})+o(1))n^2}=o(2^{0.6n^2})$.

    Suppose that a subset $S$ of $A$ contains a point $(a,b)$ satisfying $a+b\geq 1.7n$. Let $$R_2=\{(x,y)\in \mathbb{R}^2:x,y\leq 0.7n, 0.8n\leq x+y\leq0.85n\}.$$ Again, as above, since $R_2$ and $(a,b)-R_2$ are contained in the big stripe, we can find at least $\frac{\lambda(R_2)}{2}$ disjoint pairs of the form $\{(x,y), (a,b)-(x,y)\}$ inside the big stripe so that $S$ can have at most one element from each pair. So, similarly as above, the number of such $S\subseteq A$ is at most $2^{(0.6+2\gamma+\beta)n^2}\cdot(\frac{3}{4})^{\frac{\lambda(R_2)}{2}}$. Noting that $R_2$ has a larger area than $R_1$, we can take a union bound similarly as above to get that there are $o(2^{0.6n^2})$ sum-free subsets of $[n]^2$ that have a point with coordinate sum at least $1.7 n$.
\end{proof}

\noindent\textbf{Note.} If we replace the number $0.1$ in Proposition \ref{proximity} with any smaller $\alpha>0$, our proof can be adapted slightly so it still goes through. We fixed $\alpha=0.1$ as it suffices for our purposes to be concrete about the other constants in the proof. With this modification, one obtains the following typical structure result.

\begin{prop}
    The number of sum-free subsets of $[n]^2$ that are not $o(1)$-close to the big stripe is $o(2^{0.6n^2})$. That is, asymptotically almost surely, a uniformly chosen sum-free subset of $[n]^2$ is $o(1)$-close to the big stripe.
\end{prop}

\section{Counting sum-free sets close to the big stripe}
\label{far}

Our goal in this section is to count the number of sum-free subsets of $[n]^2$ that are $0.1$-close to the big stripe. To start with, we set up some notation for the region of interest.
\begin{definition}
    For $n\in \mathbb{N}$, let 
    $R_0=\{(x,y)\in [0,n]^2:0.7n\leq x+y\leq 1.7n\}$.
\end{definition}

\noindent Additionally, recall the terminology on lattice points established at the beginning of Section \ref{subsection proof of close}. Abusing notation, for $R\subset\mathbb{R}^2$, we will write $SF(R)$ to mean $SF(\Lambda(R))$.

\begin{prop}
\label{dircount}
    $SF(R_0)\leq 2^{0.6n^2 +O(n)}$.
\end{prop}
\noindent As we noted in the introduction, Theorem \ref{mainthm} immediately follows from Propositions \ref{dircount} and \ref{proximity}.

\subsection{A different region}

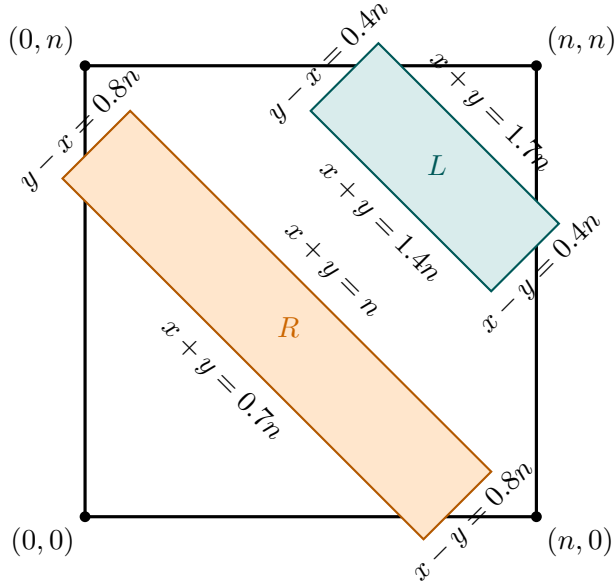
\begin{figure}[h]
    \centering
    \begin{tikzpicture}[scale=6]
  \def\n{1}

  \draw[very thick] (0,0)--(\n,0)--(\n,\n)--(0,\n)--cycle;

  \pgfmathsetmacro{\cA}{0.7*\n}
  \pgfmathsetmacro{\cB}{1.0*\n}
  \pgfmathsetmacro{\dPlus}{0.8*\n}
  \pgfmathsetmacro{\dMinus}{-0.8*\n}

  \coordinate (R1) at ({(\cA + \dMinus)/2},{(\cA - \dMinus)/2});
  \coordinate (R2) at ({(\cA + \dPlus)/2},{(\cA - \dPlus)/2});
  \coordinate (R3) at ({(\cB + \dPlus)/2},{(\cB - \dPlus)/2});
  \coordinate (R4) at ({(\cB + \dMinus)/2},{(\cB - \dMinus)/2});

  \fill[orange!20] (R1)--(R2)--(R3)--(R4)--cycle;
  \draw[thick,orange!70!black]
    (R1)--(R2)--(R3)--(R4)--cycle;

  \pgfmathsetmacro{\cC}{1.4*\n}
  \pgfmathsetmacro{\cD}{1.7*\n}
  \pgfmathsetmacro{\dTwoPlus}{0.4*\n}
  \pgfmathsetmacro{\dTwoMinus}{-0.4*\n}

  \coordinate (L1) at ({(\cC + \dTwoMinus)/2},{(\cC - \dTwoMinus)/2});
  \coordinate (L2) at ({(\cC + \dTwoPlus)/2},{(\cC - \dTwoPlus)/2});
  \coordinate (L3) at ({(\cD + \dTwoPlus)/2},{(\cD - \dTwoPlus)/2});
  \coordinate (L4) at ({(\cD + \dTwoMinus)/2},{(\cD - \dTwoMinus)/2});

  \fill[teal!15] (L1)--(L2)--(L3)--(L4)--cycle;
  \draw[thick,teal!70!black]
    (L1)--(L2)--(L3)--(L4)--cycle;

  \draw[thin,orange!70!black]
    (R1)--(R2)
      node[midway,sloped,below,yshift=-3pt,text=black] {$x+y=0.7n$};
  \draw[thin,orange!70!black]
    (R3)--(R4)
      node[midway,sloped,above,yshift=3pt,text=black] {$x+y=n$};
  \draw[thin,orange!70!black]
    (R2)--(R3)
      node[midway,sloped,below,text=black] {$x-y=0.8n$};
  \draw[thin,orange!70!black]
    (R1)--(R4)
      node[midway,sloped,above,text=black] {$y-x=0.8n$};

  \draw[thin,teal!70!black]
    (L1)--(L2)
      node[midway,sloped,below,yshift=-3pt,text=black] {$x+y=1.4n$};
  \draw[thin,teal!70!black]
    (L3)--(L4)
      node[midway,sloped,above,yshift=3pt,text=black] {$x+y=1.7n$};
  \draw[thin,teal!70!black]
    (L2)--(L3)
      node[midway,sloped,below,text=black] {$x-y=0.4n$};
  \draw[thin,teal!70!black]
    (L1)--(L4)
      node[midway,sloped,above,text=black] {$y-x=0.4n$};

  \node[text=orange!80!black,font=\bfseries] at (0.45*\n,0.42*\n) {$R$};
  \node[text=teal!70!black,font=\bfseries] at (0.78*\n,0.78*\n) {$L$};

  \foreach \p/\pos/\txt in {
    (0,0)/below left/{(0,0)},
    (\n,0)/below right/{(n,0)},
    (\n,\n)/above right/{(n,n)},
    (0,\n)/above left/{(0,n)}
  }{
    \fill \p circle (0.012);
    \node[\pos] at \p {\(\txt\)};
  }
\end{tikzpicture}
    \caption{The simpler region $R \cup L$}
    \label{fig:2}
\end{figure}

Our first lemma will reduce the problem of counting sum-free sets of $R_0$ to that of a simpler region (see Figure \ref{fig:2}). We start by defining these key regions.

\begin{definition}
    For $n\in \mathbb{N}$, define the regions $R$ and $L$ as follows:
    \begin{align*}
        &R=\{(x,y)\in \mathbb{R}^2:0.7n\leq x+y\leq n, |x-y|\leq 0.8n\}\\
        &L=\{(x,y)\in \mathbb{R}^2:2\lceil0.7n\rceil\leq x+y\leq 1.7n, |x-y|\leq 0.4n\}.
    \end{align*}
\end{definition}

\begin{lemma}
\label{proprobust}
    \[
    \sf(R_0)\leq 2^{0.36n^2+O(n)}\sf(R\cup L).
    \]
\end{lemma}

\noindent\textit{Proof.}
    Note that
    \begin{align*}
        \sf(R_0)&\leq \sf(R_0\setminus(R\cup L))\sf(R\cup L)\\
        &\leq 2^{\lambda(R_0\setminus(R\cup L))}\sf(R\cup L)\\
        &=2^{\area(R_0\setminus(R\cup L))+O(n)}\sf(R\cup L)\\
        &= 2^{0.36n^2+O(n)}\sf(R\cup L),
    \end{align*}
    where we used Lemma \ref{estarea} in the penultimate step.
\qed

\subsection{Decomposing into fibers}

In order to count the sum-free sets of the two-dimensional region $R\cup L$, we will divide the region into triplets of one-dimensional fibers and consider each triplet separately.

\begin{definition}
    Let $w\defeq\lfloor0.4n\rfloor$. Define the fibers $R_j$ for $|j|\leq 2w$ and $L_k$ for $|k|\leq w$ as follows:
    \begin{align*}
        R_j&\defeq\{(x,y)\in \Lambda(R):x-y=j\}\\
        L_k&\defeq\{(x,y)\in \Lambda(L):x-y=k\}.
    \end{align*}
\end{definition}

We label the points on a fiber by their height as follows.

\begin{definition}[height $h$]
    Define $h:R\cup L\to \mathbb{Z}_{\geq0}$ as $h((x,y))\defeq \lfloor\frac{x+y-\lceil0.7n\rceil}{2}\rfloor$ for $(x,y)\in R$ and $h((x,y))\defeq \lfloor\frac{x+y-2\lceil0.7 n\rceil}{2}\rfloor$ for $(x,y)\in L$.
\end{definition} 

Next, we define the set of Schur triples that help us partition the fibers.

\begin{definition}
    Recall $w=\lfloor0.4n\rfloor$. Define the set of triples $\mathcal{T}$ as \[
    \mathcal{T}\defeq\{(-t,2t,t):t\in [w]\}\cup \{(-w-t,2t-1,-w-1+t):t\in [w]\}.
    \]
\end{definition}
Note that $R_i, R_j, L_k$ for $(i,j,k)\in\mathcal{T}$ together partition all but $O(n)$ points of $R\cup L$. We have thus obtained a decomposition of $R\cup L$ into (Schur) triples of 1-dimensional fibres\footnote{Here's an \href{https://www.desmos.com/calculator/ezlzw2jqvb}{animation} which visually illustrates this decomposition.}.

Finally, we define the notion of a \emph{Schur embedding}, which allows us to reduce any triplet of fibres to $\{1,3,4\}\times [0.15n]_0$. Here, and in the rest of the paper, by $[n]_0$ we mean $[n]\cup \{0\}$.

\begin{definition}
    We say that $R_1$ Schur embeds into $R_2$ if there exists an injection $f:R_1\to R_2$ such that for all $a,b,c\in R_1$, $a+b=c$ if and only if $f(a)+f(b)=f(c)$.
\end{definition}

\begin{lemma}
\label{lemschurembedding}
    Suppose $i\neq j$ and $k=i+j$ satisfy $|i|,|j|\leq 2w$ and $|k|\leq w$. Then $R_i\cup R_j\cup L_k$ Schur embeds into $\{1,3,4\}\times [0.15n+O(1)]_0$.
\end{lemma}
\begin{proof}
    We construct the Schur embedding as follows. Map the fiber $R_i$, $R_j$, and $L_k$ into the fibers $x=1$, $x=3$, and $x=4$ respectively. The $y$-coordinates of the mapped points are given by the height function $h$ on $R_i$ and $R_j$, and by either $h$ or $h-1$ on $L_k$, depending on the parity of $i$ and $j$.
\end{proof}

We are now ready to provide an upper bound $\sf(R\cup L)$ in terms of the following simple function $s$.

\begin{definition}
    \[
    s(n)\defeq\sf(\{1,3,4\}\times [n])
    \]
\end{definition}

\begin{lemma}
\label{bound RUL by s}
    \[
    \sf(R\cup L)\leq 2^{O(n)}s(\lfloor0.15n\rfloor)^{0.8n}
    \]
\end{lemma}
\begin{proof}
    As $R_i, R_j, L_k$ for $(i,j,k)\in\mathcal{T}$ together partition all but $O(n)$ points of $R\cup L$, we have that
    \[
    \sf(R\cup L)\leq 2^{O(n)}\prod_{(i,j,k)\in \mathcal{T}}\sf(R_i\cup R_j \cup L_k).
    \]
    By Lemma \ref{lemschurembedding}, $R_i\cup R_j \cup L_k$ Schur embeds into $\{1,3,4\}\times [0.15n+O(1)]$ and so
    \[
    \sf(R_i\cup R_j \cup L_k)\leq \sf (\{1,3,4\}\times [0.15n+O(1)]) \leq 2^{O(1)} s(\lfloor0.15n\rfloor).
    \]
    
\end{proof}

\subsection{Bounding $\sf(\{1,3,4\}\times [n])$}
\label{sec bounding s(n)}

By virtue of Lemmas \ref{proprobust} and \ref{bound RUL by s}, any bound on $s(n)$ will translate to a bound on $\sf(R_0)$. So, in this subsection, we attempt to bound $s$. We do this by bounding the following important function.

\begin{definition}
    Let $g:\mathbb{Z}\to \mathbb{R}$ be the function satisfying
    $$g(n)\defeq \sum_{\substack{0\in S_1\subseteq [n]_0\\
                    0\in S_2\subseteq [n]_0\\ n+1\notin S_1+S_2}} 2^{-|(S_1+S_2)\cap[n]_0|}$$
    for all $n\in\mathbb{N}$ and define $g(n)\defeq 1$ for $n\in \mathbb{Z}\setminus\mathbb{N}$.
\end{definition}

\noindent $g(n)$ arises naturally because of the following lemma.

\begin{lemma}
\label{bounding s by g}
$\displaystyle s(n)\leq2^{2n+O(1)}\sum_{a,b,c\in [n]}2^{-|c-(a+b)|}g(c-(a+b)-1)$.
\end{lemma}

\begin{proof}
    For a sum-free subset $S$ of $\{1,3,4\}\times [n]_0$, let $a$ and $b$ denote the minimum elements such that $(1,a)\in S$ and $(3,b)\in S$, and let $c$  denote the maximum element such that $(4,c)\in S$. It suffices to show that the number of $S$ with given $a$, $b$ and $c$ is at most $2^{2n-|c-(a+b)|+O(1)}g(c-(a+b)-1)$.

    There are at most $2^{n-a}\cdot2^{n-b}\cdot2^c\cdot2^{O(1)}$ choices for $S$, so in the case that $c\leq (a+b)$, we are done. We can, therefore, assume that $c>a+b$. Let $\lambda=c-(a+b)-1$.
    
    Fix the subsets $S_1$ and $S_2$ of $[\lambda]_0$ such that $1\times (S_1+a)=S\cap (1\times [a,a+\lambda])$ and $3\times (S_2+b)=S\cap (3\times [b,b+\lambda])$. Note that $0\in S_1\cap S_2$ as $(1,a),(3,b)\in S$ and that $\lambda+1\notin S_1+S_2$ as $S$ is sum-free and $(4,c)\in S$.

    Having fixed $S_1$ and $S_2$, there are $2^{2n-(a+b)-2\lambda+O(1)}$ choices for $S\cap (\{1,3\}\times [n])$. As $|S_1+S_2\cap [\lambda]_0|$ points are ruled out in $S\cap (4\times [c-\lambda, c])$, there are $2^{c-|(S_1+S_2)\cap[\lambda]_0|+O(1)}$ choices for $S\cap (4\times [n])$. Summing over possible $(S_1, S_2)$, we get that there are at most $2^{2n-\lambda+O(1)}g(\lambda)$ choices for $S$, as desired.
\end{proof}

We prove the following bound on $g(n)$ in the next subsection.

\begin{lemma}
\label{bounding g}
    There exists $\eps>0$ so that $g(n)=O((2-\eps)^n)$.
\end{lemma}

\noindent As corollaries, we can obtain the following bounds on $s$ and $\sf([n]^2)$.
\begin{corollary}
\label{bound on s(n)}
    $s(n)=O(n^22^{2n})$.
\end{corollary}
\begin{proof}
    For fixed integers $a$ and $b$, $\displaystyle\sum_{c\in \mathbb{Z}}2^{-|c-(a+b)|}g(c-(a+b)-1)$ is an absolute constant by Lemma \ref{bounding g} and so, when this is summed over $a,b\in[n]$, the total is $O(n^2)$. The result now follows from Lemma \ref{bounding s by g}.
\end{proof}
\begin{corollary}
    $\sf([n]^2)=2^{0.6n^2+O(n\log n)}$.
\end{corollary}
\begin{proof}
    By Lemmas \ref{proprobust}, \ref{bound RUL by s} and Corollary \ref{bound on s(n)}, it follows that $\sf(R_0)\leq 2^{0.6n^2+O(n\log n)}$. The result now follows from Proposition \ref{proximity}.
\end{proof}

We wish to improve the error in the exponent from $O(n\log n)$ to $O(n)$. In order to do so using the strategy above, we would need $s(n)=2^{2n+O(1)}$. Unfortunately, this does not hold and in fact, $s(n)=\Theta(n^22^{2n})$ as can be seen by counting subsets of $1\times [a,n]\cup 3 \times [b,n]\cup 4\times [1,a+b)$ for the $\Theta(n^2)$ choices for $a,b \in [\frac{n}{2}]$. This means that we cannot ignore the interaction between distinct triplets of fibers. We will return to this point in Section \ref{heightprofile} but first we prove the claimed bound on $g(n)$.

\subsection{Proof of Lemma \ref{bounding g}: Bounding $g(n)$. }

The following quick observation is crucial to the proof.

 \begin{observation}
 \label{unioninsum}
     If $0\in S_1\cap S_2$, then $S_1\cup S_2 \subseteq S_1+S_2$.
 \end{observation}

We will divide the sum $g(n)$ into the sums $\Tilde{g}(n)$ and $g_k(n)$ defined below.

 \begin{definition}[$\Tilde{g}(n)$]
     Let $\Tilde{g}(n)$ denote the sum of $2^{-|(S_1+S_2)\cap[n]_0|}$ over all subsets $S_1$ and $S_2$ of $[n]_0$ so that $0\in S_1 \cap S_2$, $n+1\notin S_1+S_2$ and $k\in S_1+S_2$ for all integers $k$ satisfying $\frac{3n}{4}+1\leq k\leq \frac{15n}{16}$.
 \end{definition}

  \begin{definition}[$g_k(n)$]
     For a positive integer $k$ satisfying $\frac{n}{16}\leq k\leq \frac{n}{4}$, let $g_k(n)$ denote the sum of $2^{-|(S_1+S_2)\cap[n]_0|}$ over all subsets $S_1$ and $S_2$ of $[n]_0$ so that $0\in S_1 \cap S_2$, $n+1\notin S_1+S_2$ and $n+1-k\notin S_1+S_2$.
 \end{definition}

 \begin{lemma}
 \label{calc1}
    There exists $\eps>0$ so that $\Tilde{g}(n)=O((2-\eps)^n)$.
 \end{lemma}
 \begin{proof}
    From the definition of $\Tilde{g}$ and Observation \ref{unioninsum}, it follows that
     \begin{align*}
         \Tilde{g}(n)\leq \sum_{\substack{0\in S_1\subseteq [n]_0\\
                    0\in S_2\subseteq [n]_0\\ n+1\notin S_1+S_2}} 2^{-|S_1\cup S_2\cup [\frac{3n}{4}+1,\frac{15n}{16}]|}\eqdef r(n)\text{, say.}
    \end{align*}
    Note that $2r(n)$ factorises into the product
    \begin{align*}
            \Bigg(\prod_{t\in [\frac{3n}{4}+1,\frac{15n}{16}]}\sum_{\substack{S_1\subseteq \{t,n+1-t\}\\
                    S_2\subseteq \{t,n+1-t\}\\ n+1\notin S_1+S_2}}2^{-|S_1\cup S_2\cup \{t\}|}\Bigg)\Bigg(\prod_{\substack{t\in (\frac{n+1}{2},\frac{3n}{4}+1)\\ \text{or }t\in(\frac{15n}{16},n]}}\sum_{\substack{S_1\subseteq \{t,n+1-t\}\\
                    S_2\subseteq \{t,n+1-t\}\\ n+1\notin S_1+S_2}}2^{-|S_1\cup S_2|}\Bigg)
    \end{align*}
    which is at most
    \begin{align*}
            (4\cdot2^{-1}+5\cdot2^{-2})^{\frac{3n}{16}}(2\cdot2^{-2}+6\cdot2^{-1}+1)^{\frac{5n}{16}+3}=O((2-\eps)^n).
     \end{align*}
 \end{proof}

 To bound $g_k(n)$, it will be helpful to consider some sums over independent sets in bipartite graphs. Let us start by setting up some notation. A bipartite graph $G$ with parts $V_1$ and $V_2$ is given by its set of vertices $V_1\sqcup V_2$ and set of edges $E(G)\subset V_1\times V_2$. In particular, we allow common labels between the two parts. For a bipartite graph $G$ with parts $V_1$ and $V_2$, let $I(G)$ denote the set of independent sets in $G$, that is, $I(G)\defeq\{S_1\sqcup S_2:S_1\subseteq V_1, S_2\subseteq V_2\text{ and there is no edge between }S_1\text{ and }S_2\}$.

\begin{definition}
    For a bipartite graph $G$ with parts $V_1$ and $V_2$, let $\sfr(G)$ denote the sum of $2^{-|S_1\cup S_2|}$ over all subsets $S_1\subset V_1$ and subsets $S_2\subset V_2$ such that $S_1\sqcup S_2\in I(G)$. We write this sum as\[\sfr(G)=\sum_{S_1\sqcup S_2\in I(G)}2^{-|S_1\cup S_2|}.\]
\end{definition}

\noindent Note that the union in the exponent of the summand in the definition of $\sfr$ is not disjoint. In particular, this means that $\sfr(G)$ is not just an evaluation of the independence polynomial of $G$.

We will need the following lemma to bound $g_k$.

\begin{lemma}
\label{comput}
    Let $G_n$ denote the union of a path on $n$ vertices and its vertex dsijoint \emph{reflection}, that is, $V(G_n)=[n]\sqcup [n]$, and $E(G_n)=\{(a,b):|a-b|=1\}$. Then, $\sfr(G_n)< 2^n$ for all $n\geq 8$.
\end{lemma}
\begin{proof}
Let $\tilde{G_n}$ denote the bipartite graph with $V(\tilde{G}_n)=[n]\sqcup[n-1]$, and $E(\tilde{G}_n)=\{(a,b):|a-b|=1\}$. Then, notice that
\begin{align}
\label{recursion}
\begin{split}
    &\sfr(G_n)=2^{-1}\sfr(G_{n-2})+\sfr(G_{n-1})+\sfr(\tilde{G}_{n-1})\\
    & \sfr(\tilde{G}_n)=\sfr(G_{n-1})+2^{-1}\sfr(\tilde{G}_{n-1}).
\end{split}
\end{align}

With the initial values $\sfr(G_0)=1$, $\sfr(G_1)=2.5$ and $\sfr(\tilde{G}_{1})=1.5$, one can calculate $\sfr(G_8)=250.5\dots$, $\sfr(G_9)=487.8\dots$, $\sfr(\tilde{G}_9)=337.0\dots$ explicitly. We can now prove by induction that $$\sfr(G_n)<2^n\ \forall n\geq 8\text{ and }\sfr(\tilde{G}_n)<\frac{2}{3}\cdot2^n\ \forall n\geq 9,$$ with $n=8$ and $9$ serving as base cases and the induction step following from the recursion \ref{recursion}.

\end{proof}

We are now ready to execute the final step in our proof, namely, bounding $g_k(n)$.

 \begin{lemma}
 \label{calc2}
 There exists $\eps>0$ so that for $k\in[\frac{n}{16}, \frac{n}{4}]$, it holds that $g_k(n)=O((2-\eps)^n)$.
 \end{lemma}
\begin{proof}
    Construct the bipartite graph $H=H(k)$ with
    $$V(H)=([n],[n])\text{ and }E(H)=\{ab: a+b=n+1\text{ or } a+b=n+1-k\}.$$

    From the definition of $g_k$ and Observation \ref{unioninsum}, it follows that
     \begin{align*}
         g_k(n)\leq \sum_{(S_1,S_2)\in I(H)} 2^{-|S_1\cup S_2|}=\sfr(H)
    \end{align*}
    
    For $a>n-k$, there is a maximal path $a,b,a-k,b+k,\cdots,a-(q-1)k,b+(q-1)k$ in $H$, where $b=n+1-a$ and $q=\lfloor\frac{a-1}{k}\rfloor+1$. In fact, $H$ decomposes into such paths. The bounds on $k$ imply that $q$ lies in $[4,17]$.
    
    There are at most $17^2$ maximal paths where a vertex and its reflection both lie on the same path. So, apart from $O(1)$ vertices, $H$ decomposes into $H_1,\dots, H_l$, where $H_i$ composed of an alternating path on $2q_i\in[8,34]$ vertices and its vertex-disjoint reflection, for all $i\in [l]$. Hence, Lemma $\ref{comput}$ applies to the $H_i$ and we have
    \begin{align*}
        \sfr(H) &\leq 2^{O(1)}\prod_{i=1}^l\sfr(H_i)\leq 2^{O(1)}\prod_{i=1}^l (2-\eps)^{2q_i}=O((2-\eps)^n).
    \end{align*}
\end{proof}

Finally, combining Lemmas \ref{calc1} and \ref{calc2}, we see that $g(n)=O((2-\eps)^n)$ for some $\eps>0$.

\subsection{Height Profile}
\label{heightprofile}

As we remarked at the end of Section \ref{sec bounding s(n)}, we must take into account the interaction between distinct triplets of fibers in order to prove Proposition \ref{dircount}. We will do this as follows. We start by fixing the \emph{height profile} of a sum-free subset $S\subseteq R \cup L$, that is, we specify the element of $S$ of minimum height along each fiber of $R$, and that of maximum height along each fiber of $L$. We introduce the key notion of \emph{discrepancy} for height profiles and show that:
\begin{enumerate}[label=(\alph*)]
    \item there are few sum-free subsets with height profiles of large discrepancy and
    \item height profiles with small discrepancies are almost linear and hence small in number.
\end{enumerate}

\begin{definition}
    For a sum-free subset $S\subseteq R\cup L$, define the height profiles $m:[-2w,2w]\to \mathbb{N}\cup \{0\}$ and $M:[-w,w]\to \mathbb{N}\cup\{0\}$ as follows. For $i\in [-2w,2w]$ let $m(i)$ be the minimum height of an element in $S\cap R_i$. For $k\in [-w,w]$, let $M(k)$ be the maximal height of an element in $S\cap L_k$.
\end{definition}

\begin{definition}
    For $R'\subseteq R\cup L$, $S'$ is a sum-free subset of $R'$ \emph{respecting} the height profile $(m,M)$ if there exists a sum-free subset $S\subseteq R\cup L$ with height profile $(m,M)$ so that $S'=S\cap R'$.
\end{definition}

We expect that in a typical sum-free set, for most $(i,j,k)\in \mathcal{T}$, $|M(k)-m(j)-m(i)|$ is small. Further, this should also hold when $\mathcal{T}$ is replaced by other Schur triples of fibers. We define the following notion of discrepancy to capture departure from this behaviour.

\begin{definition}
    For $s\in\{-1,0,1\}$ and $(i,j,k)\in\mathcal{T}$, let $d_s(i)\defeq M(k+s)-m(j+s)-m(i)$ and let $D_s\defeq\sum_{(i,j,k)\in\mathcal{T}}|d_s(i)|$,
    where the summand is taken to be zero when it is not defined. Finally, we define the discrepancy $D$ as $D\defeq\max\{D_{-1},D_0,D_1\}$.
\end{definition}

\begin{lemma}
\label{disc1}
    There exists an absolute constant $\alpha>0$ so that the following holds. Suppose $i\neq j$ and $k=i+j$ satisfy $|i|,|j|\leq 2w$ and $|k|\leq w$. For any $(m,M)$, the number of sum-free subsets of $R_i\cup R_j\cup L_k$ respecting the height profile $(m,M)$ is at most $2^{0.3n-\alpha|M(k)-m(j)-m(i)|+O(1)}$.
\end{lemma}

\begin{proof}
    Consider the Schur embedding of $R_i\cup R_j\cup L_k$ into $\{1,3,4\}\times [0.15n+O(1)]_0$ given by Lemma \ref{lemschurembedding}. Notice that sum-free subsets of the former which respect $(m,M)$ correspond to sum-free subsets of the latter with minimum elements $a=m(i)$ and $b=m(j)$ along the fibres $x=1$ and $x=3$ respectively and maximal element $c$ along the fiber $x=4$, where $c$ equals $M(k)$ or $M(k)-1$. Therefore, as in the proof of Lemma \ref{bounding s by g}, we have that the number of such sum-free sets is at most $2^{2\times(0.15n)-|c-(a+b)|+O(1)}g(c-(a+b)-1)$, which is at most $2^{0.3n-\alpha|c-(a+b)|+O(1)}$ for some $\alpha>0$ by Lemma \ref{bounding g}. Noting that $|c-(a+b)|=|M(k)-m(j)-m(i)|+O(1)$, we are done.
\end{proof}

\begin{corollary}
\label{corollary large disc}
    There exists an absolute constant $\alpha>0$ such that the following holds. Let $(m,M)$ be a height profile with discrepancy $D$. Then, the number of sum-free subsets of $R\cup L$ with height profile $(m,M)$ is at most
    \[
    2^{0.24n^2-\alpha D+O(n)}.
    \]
\end{corollary}

\begin{proof}
    Let $s\in\{-1,0,1\}$. As $R_i, R_{j+s}, L_{k+s}$ for $(i,j,k)\in \mathcal{T}$ together partition all but $O(n)$ points of $R\cup L$, we have by Lemma \ref{disc1} that the number of sum-free subsets of $R\cup L$ with height profile $(m,M)$ is at most \[2^{O(n)}\prod_{(i,j,k)\in\mathcal{T}}2^{0.3n-\alpha|M(k)-m(j)-m(i)|+O(1)}=2^{0.24n^2-\alpha D_i+O(n)}.\]
\end{proof}

We now show that small discrepancy implies almost linearity for $m$. Let $\Delta$ denote the forward difference operator, that is, $(\Delta f)(i)=f(i+1)-f(i)$.

\begin{lemma}
\label{disc2}
    \[
    \sum_{i=-2w}^{2w-2}|\Delta^2m(i)|=O(D+n)
    \]
\end{lemma}
\begin{proof}
    Note first that we have the uniform bound $|\Delta^2m(i)|=O(n)$, so we can ignore $O(1)$ values in the range of the sum.
    For $t\in [w-1]$, we have that
    \begin{align*}
        \Delta^2m(2t)&=-d_0(-t)-d_0(-t-1)+d_{-1}(-t-1)+d_1(-t)\\
        \Delta^2m(2t-1)&=-d_0(-w-t-1)-d_0(-w-t)+d_{-1}(-w-t-1)\\
        &\quad+d_1(-w-t)\\
        \Delta^2m(-t-1)
        &=2d_0(-t)+d_0(-w-t)+d_0(-w-t-1)-d_{-1}(-t-1)\\
        &\quad-d_{-1}(-w-t-1)-d_1(-t+1)-d_1(-w-t)\\
        \Delta^2m(-w-t-1)
        &=d_0(-t)+2d_0(-w-t)+d_0(-t+1)-d_{1}(-t+1)\\
        &\quad -d_{-1}(-w-t-1)-d_{-1}(-t)-d_1(-w-t+1).
    \end{align*}
    Applying the triangle inequality and summing over $t$, we are done.
\end{proof}

As a corollary, we obtain that there are few choices for $(m,M)$ with $D$ small.

\begin{lemma}
\label{lemma number of height profile}
    For any $\eps>0$, there are at most $2^{\eps D+O_{\eps}(n)}$ choices for the height profile $(m,M)$ given $D$.
\end{lemma}
\begin{proof}

    By Lemma \ref{disc2}, $\sum_{i}|\Delta^2m(i)|=O(D+n)$, and so the number of choices for the sequence $|\Delta^2m(i)|$ given $D$ is $O(D+n)\binom{O(n+D+w)}{4w-2}\leq \left(\frac{O(n+D)}{w}\right)^{4w}\leq 2^{O\left(n\log\left(\frac{D}{n}+1\right)+n\right)}$. Further, the number of choices for the signs of $\Delta^2m(i)$ is $2^{O(n)}$. Along with $m(2w-1)$ and $m(2w)$, the sequence $\Delta^2m(i)$ determines $m$ completely and so the number of choices for $m$ is $2^{O\left(n\log\left(\frac{D}{n}+1\right)+n\right)}$.

    Once $m$ is determined, the number of choices for $M$ is at most $\prod_{(i,j,k)\in \mathcal{T}} (2d_0(i)+10)\leq 2^{O\left(n\log\left(\frac{D}{n}+1\right)+n\right)}$, by the concavity of $\log$. Finally, note that $n\log\left(\frac{D}{n}+1\right)\leq \eps D+Kn$ for large enough $K=K(\eps)$.
\end{proof}

\begin{proof}[Proof of Proposition \ref{dircount}]
    Let $\alpha$ be the absolute constant given by Corollary \ref{corollary large disc}. We apply Lemma \ref{lemma number of height profile} with $\eps\defeq \frac{\alpha}{2}$ to obtain that the number of sum-free subsets of $R\cup L$ with discrepancy $D$ is at most $2^{0.24n^2+O(n)-\eps D}$. Summing over $D\geq 0$, we get that $SF(R\cup L)\leq 2^{0.24n^2+O(n)}$ and so $SF(R_0)\leq 2^{0.6n^2+O(n)}$ by Lemma \ref{proprobust}.
\end{proof}

\printbibliography

\end{document}